\documentclass[11pt,reqno,a4paper,twoside]{extarticle}

\usepackage[osf]{newtxtext}

\usepackage[left=3cm,right=3cm,top=2.2cm,bottom=2.2cm]{geometry}

\usepackage{graphicx, varioref, amsfonts}
\usepackage{amsthm,amssymb,amsmath,mathrsfs,mathtools,stmaryrd} %\usepackage{stmaryrd}
\usepackage{dsfont}
\usepackage{upgreek} %% txgreeks
\usepackage{esint} %% nice integral symbols 
\usepackage{breakurl}
\usepackage{empheq}
\usepackage{enumerate,enumitem}
\usepackage{hyperref}
\usepackage{cleveref}
\usepackage{soul}
\usepackage{cancel}
\usepackage{amscd}
\usepackage{srcltx}
\usepackage{ifthen}
\usepackage{float}
\usepackage{color}
\usepackage{comment}
\usepackage{chngcntr}
\usepackage{etoolbox}
\usepackage{fancyhdr}
\usepackage{breqn}
\usepackage{cite}
\usepackage[T1]{fontenc}%spelling swiech

\usepackage{titlefoot}

\usepackage{authblk} % For \affil command
\theoremstyle{plain}
\newtheorem{lemma}{Lemma}[section]

\newtheorem{proposition}[lemma]{Proposition}

\newtheorem{theorem}[lemma]{Theorem}

\newtheorem{assumption}[lemma]{Assumption}

\theoremstyle{definition}

\newtheorem{definition}[lemma]{Definition}

\newtheorem{remark}[lemma]{Remark}
\newlist{todolist}{itemize}{2}
\setlist[todolist]{label=$\square$}
\usepackage{pifont}

\begin{document}
	
	\title{Semilinear Feynman--Kac Formulae for $B$-Continuous Viscosity Solutions}
	\newcommand\shorttitle{Semilinear Feynman--Kac Formulae for $B$-Continuous Viscosity Solutions}
	\date{November 21, 2024}
	%\date{\today}
	
	\author{Lukas Wessels}
	\newcommand\authors{}
	
	\affil{\small School of Mathematics, Georgia Institute of Technology}

\maketitle

\unmarkedfntext{\textit{Mathematics Subject Classification (2020) ---} 60H30, 35R15, 60H15}

%35K57: Reaction-diffusion equations
%35R15: PDEs on infinite-dimensional (e.g., function) spaces (= PDEs in infinitely many variables) 
%49K27: Optimality conditions for problems in abstract spaces
%49K45: Optimality conditions for problems involving randomness
%49L12:	Hamilton-Jacobi equations in optimal control and differential games
%49L20: Dynamic programming in optimal control and differential games
%49L25:	Viscosity solutions to Hamilton-Jacobi equations in optimal control and differential games
%49N35:	Optimal feedback synthesis
%60H15: SPDEs (aspects of stochastic analysis)
%60H30: Applications of stochastic analysis (to PDEs, etc.)
%65K10: Numerical optimization and variational techniques
%93E20: Optimal stochastic control

\unmarkedfntext{\textit{Keywords and phrases ---} backward stochastic differential equations, partial differential equations in infinite dimensional spaces, $B$-contnuous viscosity solutions}

\unmarkedfntext{\textit{Email}: \textbullet$\,$ wessels@gatech.edu}

\begin{abstract}
	We prove the existence of a $B$-continuous viscosity solution for a class of infinite dimensional semilinear partial differential equations (PDEs) using probabilistic methods. Our approach also yields a stochastic representation formula for the solution in terms of a scalar-valued backward stochastic differential equation. The uniqueness is proved under additional assumptions using a comparison theorem for viscosity solutions. Our results constitute the first nonlinear Feynman--Kac formula using the notion of $B$-continuous viscosity solutions and thus introduces a framework allowing for generalizations to the case of fully nonlinear PDEs.
\end{abstract}

%\tableofcontents

\section{Introduction}

The classical Feynman--Kac formula going back to Richard Feynman \cite{feynman1948} and Mark Kac \cite{kac1949} gives a stochastic representation for the solution of a linear partial differential equation (PDE) in terms of a path integral over a diffusion process. This result and generalizations thereof have since found many applications in various fields such as quantitative finance and stochastic optimal control, see e.g. \cite{pham2009}. Recent years have shown a rising interest in control problems associated with stochastic partial differential equations (SPDEs), see e.g. \cite{fabbri2017}. Therefore, generalizations of the Feynman--Kac formula to infinite dimensions are of increasing importance. Another emerging field that relies on nonlinear generalizations of the Feynman--Kac formula is numerical analysis for high-dimensional PDEs. The stochastic representation for solutions of PDEs can be utilized to develop numerical schemes based on machine learning methods, see e.g. \cite{han2018,beck2019,hure2020}. In this paper, we prove a generalization of the classical Feynman--Kac formula to semilinear PDEs in infinite dimensional spaces using the theory of $B$-continuous viscosity solutions and backward stochastic differential equations (BSDEs).

For finite initial and terminal times $0\leq t< T<\infty$, respectively, and a separable Hilbert space $H$, we consider the SPDE
\begin{equation}\label{forwardequation}
	\begin{cases}
		\mathrm{d}X^{t,x}_s = [ A X^{t,x}_s + b(s,X^{t,x}_s) ] \mathrm{d}s + \sigma(s,X^{t,x}_s) \mathrm{d}W_s,\quad s\in [t,T]\\
		X^{t,x}_t =x\in H,
	\end{cases}
\end{equation}
where $A:\mathcal{D}(A)\subset H\to H$ is an unbounded linear operator, and $b:[0,T]\times H \to H$ and $\sigma:[0,T]\times H \to L_2(\Xi,H)$ are the drift and noise coefficient, respectively. Moreover, $(W_s)_{s\in [t,T]}$ is a cylindrical Wiener process on some separable Hilbert space $\Xi$ and defined on some probability space $(\Omega,\mathcal{F},(\mathcal{F}_s)_{s\in [t,T]}, \mathbb{P})$ with $(\mathcal{F}_s)$ being its natural filtration augmented by all $\mathbb{P}$-null sets. Furthermore, we consider the infinite dimensional PDE
\begin{equation}\label{PDE}
	\begin{cases}
		v_t(t,x) + \langle Ax + b(t,x),Dv(t,x) \rangle_H \\
		\quad + \frac12 \text{tr}(\sigma^{\ast}(t,x) D^2v(t,x) \sigma(t,x)) - f(t,x,v(t,x),\sigma^{\ast}Dv(t,x)) = 0, \quad (t,x)\in [0,T]\times H\\
		v(T,x) = g(x),\quad x\in H,
	\end{cases}
\end{equation}
where $f:[0,T]\times H \times \mathbb{R}\times \Xi\to \mathbb{R}$ and $g : H\to \mathbb{R}$ are given. In order to introduce the main ideas, let us assume that equation \eqref{PDE} admits a smooth solution $v$ with $Dv(t,x)\in \mathcal{D}(A^{\ast})$, where $A^{\ast}$ denotes the adjoint of $A$. Then, applying It\^o's formula and plugging in equation \eqref{PDE} immediately yields that
\begin{equation}\label{pair}
	\begin{cases}
		Y^{t,x}_s = v(s,X^{t,x}_s)\\
		Z^{t,x}_s = \sigma^{\ast}Dv(s,X^{t,x}_s)
	\end{cases}
\end{equation}
solves the BSDE
\begin{equation}\label{bsde}
	\begin{cases}
		\mathrm{d}Y^{t,x}_s = f(s,X^{t,x}_s,Y^{t,x}_s,Z^{t,x}_s) \mathrm{d}s + \langle Z^{t,x}_s ,\mathrm{d}W_s\rangle_{\Xi},\quad s\in [t,T]\\
		Y^{t,x}_T = g(X^{t,x}_T).
	\end{cases}
\end{equation}
\iffalse
Indeed, applying It\^o's formula yields
\begin{equation}
	\begin{split}
		v(T,X^{t,x}_T) =& v(t,X^{t,x}_t) + \int_t^T v_t(r,X^{t,x}_r) \mathrm{d}r\\
		& + \int_t^T \langle Dv(r,X^{t,x}_r), \mathrm{d}X^{t,x}_r\rangle + \frac12 \int_t^T D^2 v(r,X^{t,x}_r) \mathrm{d} \langle X^{t,x}\rangle_r.
	\end{split}
\end{equation}
Plugging in equation \eqref{PDE}, we obtain
\begin{equation}
	\begin{split}
		&g(X^{t,x}_T)\\
		&= v(t,X^{t,x}_t) + \int_t^T f(r,X^{t,x}_r,v(r,X^{t,x}_r),\sigma^{\ast}Dv(r,X^{t,x}_r)) \mathrm{d}r + \int_t^T \sigma^{\ast}Dv(r,X^{t,x}_r) \mathrm{d}W_r,
	\end{split}
\end{equation}
which shows that the pair $(Y^{t,x},Z^{t,x})$ given by \eqref{pair} is indeed a solution of the BSDE \eqref{bsde}. 
\fi
In particular, for $s=t$, equations \eqref{pair} and \eqref{bsde} yield the following stochastic representation of $v$:
\begin{equation}\label{stochasticrepresentation2}
	\begin{split}
	v(t,x)& = Y^{t,x}_t\\
	&= g(X^{t,x}_T) - \int_t^T f(s,X^{t,x}_s,v(s,X^{t,x}_s),\sigma^{\ast}Dv(s,X^{t,x}_s)) \mathrm{d}s\\
	&\quad  - \int_t^T \langle \sigma^{\ast}Dv(s,X^{t,x}_s), \mathrm{d}W_s\rangle_{\Xi} .
	\end{split}
\end{equation}
The preceding discussion is only formal since the PDE \eqref{PDE} in general, does not admit a smooth solution and therefore, the solution of the BSDE \eqref{bsde} can not be given in terms of the derivative of $v$ as in \eqref{pair}. However, the stochastic representation \eqref{stochasticrepresentation2} remains valid even if the solution of equation \eqref{PDE} is not differentiable. The objective of this paper is to prove that the stochastic representation
\begin{equation}\label{stochasticrepresentation}
	u(t,x):= Y^{t,x}_t
\end{equation}
yields a $B$-continuous viscosity solution for the PDE \eqref{PDE}, see Theorem \ref{mainresult}. Furthermore, we prove under additional assumptions that $u$ as given in \eqref{stochasticrepresentation} is the unique solution of the PDE \eqref{PDE}.

\begin{remark}
	The representation \eqref{stochasticrepresentation} can be viewed as a nonlinear generalization of the classical Feynman--Kac formula: If $f(s,x,y,z) = k(s,x)y + l(s,x)$ for some functions $k,l: [t,T]\times H \to \mathbb{R}$, then the BSDE \eqref{bsde} admits an explicit solution formula for $Y^{t,x}$ in terms of $X^{t,x}$ and $Z^{t,x}$. The $Z^{t,x}$-dependency drops out when taking the expectation in \eqref{stochasticrepresentation2}, yielding the classical Feynman--Kac formula, see e.g. \cite[Chapter 5, Theorem 7.6]{karatzas1991}.
\end{remark}

In the finite dimensional case, nonlinear Feynman--Kac formulae using BSDEs go back to Pardoux and Peng \cite{pardoux1992}. Following their seminal work, there were many generalizations of this relationship between PDEs and BSDEs in various directions such as relaxing the assumptions on the coefficients \cite{bally2005,kobylanski2000}, and deriving a similar relationship for integral-partial differential equations and BSDEs involving a forward jump-diffusion process \cite{barles1997}, for nonlocal PDEs and mean-field BSDEs \cite{buckdahn2009}, for quasilinear PDEs and forward-backward SDEs \cite{feng2018,pardoux1999,peng1999}, for reflected BSDEs and related obstacle problems for PDEs \cite{elkaroui1997} and, more recently, even for fully nonlinear PDEs and second order BSDEs \cite{cheridito2007,soner2012,soner2013}.

In the infinite dimensional case, the first nonlinear Feynman--Kac formula using BSDEs was obtained by Fuhrman and Tessitore in \cite{fuhrman2002}, again followed by generalizations relaxing the assumptions on the coefficients \cite{briand2008,fuhrman2005,masiero2008,masiero2014}, and deriving similar relationships for elliptic PDEs and infinite horizon BSDEs \cite{fuhrman2004,hu2015,guatteri2020}, for PDEs on a space of continuous functions and stochastic delay differential equations \cite{fuhrman2010,masiero2021}, for integro-differential equations and BSDEs involving jump-diffusions \cite{bandini2019}, and for quasilinear PDEs and forward-backward SDEs \cite{cerrai2022}.

Except for \cite{bandini2019}, all of the results in the infinite dimensional case study the PDE \eqref{PDE} in the framework of mild solutions. In \cite{bandini2019}, the authors obtain stochastic representations for viscosity solutions of PDEs in infinite dimensional spaces similar to, and in some ways more general than, the one described in \eqref{stochasticrepresentation}. However, their proofs make use of the underlying control problem, which imposes a certain structure on the nonlinearity $f$ in \eqref{PDE}. In this work, our objective is to prove such a relationship between BSDEs and PDEs in infinite dimensional spaces using the framework of $B$-continuous viscosity solutions without relying on stochastic control methods. As a byproduct, our probabilistic approach yields a new method to prove existence of $B$-continuous viscosity solutions for infinite dimensional PDEs involving unbounded terms. There are various methods in the existing literature for proving the existence of solutions, see \cite[Chapter 3]{fabbri2017} and the references therein. These existing approaches construct solutions even for fully nonlinear PDEs exploiting the connection with infinite dimensional stochastic control problems, using finite dimensional approximations, or by extending Perron's method to the infinite dimensional case. However, in our setting, we do not impose any structural assumptions on the nonlinearity $f$ and therefore, in general one cannot introduce an associated stochastic control problem to construct a solution to the PDE \eqref{PDE}. Furthermore, the method using finite dimensional approximations imposes the additional assumption that the unbounded operator $A$ satisfies the strong $B$-condition for a compact operator $B$ (see Assumption \ref{assumptionBcondition} for a definition of the strong $B$-condition). Finally, Perron's method requires the additional coercivity assumption $-\langle A^{\ast}x,x \rangle_H \geq C \|B^{-1/2}x\|_H^2$ for some constant $C>0$ and all $x\in \mathcal{D}(A^{\ast})$, see \cite[Section 3.9]{fabbri2017}.

Let us also mention that our work introduces a framework for the extension of the classical Feynman--Kac formula to the case of infinite dimensional fully nonlinear PDEs using second order BSDEs in the spirit of the corresponding results in finite dimensions, see \cite{cheridito2007,soner2012,soner2013}. This problem will be investigated in future work.

The remainder of the paper is organized as follows: In Section \ref{assumptions}, we introduce the notation and the assumptions we are working with in the following sections. In Section \ref{BSDEs}, we prove some results for scalar-valued BSDEs driven by an infinite dimensional Wiener process. Section \ref{mainresult} contains our main result, the existence of a $B$-continuous viscosity solution of equation \eqref{PDE} and its stochastic representation \eqref{stochasticrepresentation2}. Furthermore, the uniqueness is also discussed in Section \ref{mainresult}.

\section{Assumptions}\label{assumptions}

We impose the following assumptions on the unbounded operator $A$.

\begin{assumption}\label{assumptionA}
	\begin{enumerate}[label=(A\arabic*)]
		\item\label{assumptionAsemigroup} Let $A:\mathcal{D}(A) \subset H \to H$ be a linear, densely defined, maximal dissipative operator.
		\item\label{assumptionBcondition} Let $B\in L(H)$ be a strictly positive, self-adjoint operator that satisfies the strong $B$-condition for $A$, i.e., $A^{\ast} B \in L(H)$ and
		\begin{equation}
				-A^{\ast} B + c_0 B\geq I
		\end{equation}
		for some $c_0 \geq 0$.
	\end{enumerate}
\end{assumption}

Under Assumption \ref{assumptionAsemigroup}, the operator $A$ is the generator of a $C_0$-semigroup which we denote in the following by $(S(s))_{s\in [t,T]}$. Furthermore, using the operator $B$, we define the space $H_{-1}$ as the completion of the space $H$ with respect to the norm
\begin{equation}
	\|x\|^2_{H_{-1}} := \langle Bx,x\rangle_H.
\end{equation}

We impose the following assumptions on $b$ and $\sigma$.
\begin{assumption}\label{assumptionB}
	\begin{enumerate}[label=(B\arabic*)]
		\item\label{assumptionb} Let $b:[0,T]\times H\to H$ be $\mathcal{B}([0,T]) \otimes \mathcal{B}(H)/\mathcal{B}(H)$-measurable and let there be a constant $C>0$ such that
		\begin{align}
			\| b(s,x) - b(s,x^{\prime}) \|_H &\leq C \|x-x^{\prime}\|_H\\
			\| b(s,x) \|_H & \leq C (1 + \|x\|_H )
		\end{align}
		for all $s\in [0,T]$, $x,x^{\prime}\in H$.
		\item\label{assumptionsigma} Let $\sigma:[0,T]\times H \to L_2(\Xi;H)$ be $\mathcal{B}([0,T])\otimes \mathcal{B}(H)/\mathcal{B}(L_2(\Xi,H))$-measurable and let there be a constant $C>0$ such that
		\begin{align}
			\| \sigma(s,x) - \sigma(s,x^{\prime}) \|_{L_2(\Xi,H)} &\leq C \|x-x^{\prime}\|_{H_{-1}}\\
			\| \sigma(s,x) \|_{L_2(\Xi,H)} & \leq C (1 + \|x\|_H ),
		\end{align}
		for all $s\in [0,T]$, $x,x^{\prime}\in H$.
	\end{enumerate}
\end{assumption}

\begin{remark}
	While the Lipschitz condition on $b$, and the linear growth conditions on $b$ and $\sigma$ are standard assumptions in the theory of SPDEs, the Lipschitz condition on $\sigma$ with respect to the $H_{-1}$-norm imposes additional restrictions. Note, however, that this is a standard assumption in the theory of $B$-continuous viscosity solutions, see \cite[Chapter 3]{fabbri2017}. For examples in which this condition is satisfied, see \cite[Remark 3.21]{fabbri2017}. 
\end{remark}

Finally, we impose the following assumptions on the coefficients of the BSDE.
\begin{assumption}\label{assumptionC}
	\begin{enumerate}[label=(C\arabic*)]
		\item\label{assumptionf} Let $f:[0,T]\times H \times \mathbb{R}\times \Xi\to \mathbb{R}$ be uniformly continuous on bounded subsets of $(0,T)\times H\times \mathbb{R}\times \Xi$ and satisfy the following conditions:
		\begin{itemize}
			\item There exists a constant $C>0$ such that
			\begin{equation}
				|f(s,x,y,z) - f(s,x^{\prime},y^{\prime},z^{\prime})| \leq C ( \|x-x^{\prime}\|_H+ |y-y^{\prime}|+ \|z-z^{\prime}\|_{\Xi})
			\end{equation}
			for all $s\in [0,T]$, $x,x^{\prime}\in H$, $y,y^{\prime}\in \mathbb{R}$, $z,z^{\prime}\in \Xi$.
			\item It holds
			\begin{equation}
				\int_0^T |f(s,0,0,0)|^2 \mathrm{d}s < \infty.
			\end{equation}
		\end{itemize}
		\item\label{assumptiong} For $g:H\to\mathbb{R}$, let there be a constant $C>0$ such that
		\begin{equation}
			|g(x) - g(x^{\prime})| \leq C \|x-x^{\prime}\|_{H}
		\end{equation}
		for all $x,x^{\prime}\in H$.
	\end{enumerate}
\end{assumption}

\begin{remark}
	Note that the SPDE \eqref{forwardequation} as well as the BSDE \eqref{bsde} are well-posed under these assumptions, see e.g. \cite[Theorem 7.2]{daprato2014} and \cite[Proposition 6.20]{fabbri2017}, respectively.
\end{remark}

\section{BSDEs}\label{BSDEs}

In this section, we prove some basic results for scalar-valued BSDEs with infinite dimensional driving noise. For the corresponding theory when the driving noise is finite dimensional, see e.g. \cite{elkaroui1997_2,pardoux2014,zhang2017}.

\subsection{Explicit Solution of Linear BSDE}

First, we need the following explicit solution formula for linear BSDEs. Consider the equation
\begin{equation}\label{linearbsde}
\begin{cases}
	\mathrm{d}Y_s = -[ a_s Y_s + b_s + \langle c_s,Z_s\rangle_{\Xi} ]\mathrm{d}s + \langle Z_s,\mathrm{d}W_s \rangle_{\Xi}.\quad s\in [t,T]\\
	Y_T= \eta\in L^2(\Omega),
\end{cases}
\end{equation}
where $a,b:[t,T]\times\Omega \to \mathbb{R}$ and $c:[t,T]\times\Omega\to \Xi$ are progressively measurable processes and $\eta$ is $\mathcal{F}_T$-measurable. Throughout this section, we assume that $a$ and $c$ are uniformly bounded in $(s,\omega)\in [t,T]\times \Omega$, and that $b$ is square-integrable, i.e., $\mathbb{E} [ \int_t^T |b_s|^2 \mathrm{d}s ] < \infty$. Under these assumptions, the linear BSDE \eqref{linearbsde} has a unique solution, see \cite[Proposition 6.20]{fabbri2017}. Let
\begin{equation}\label{gamma}
	\Gamma_s = \exp \left [ \int_t^s a_r - \frac12 \| c_r \|_{\Xi}^2 \mathrm{d}r + \int_t^s \langle c_r, \mathrm{d}W_r \rangle_{\Xi} \right ].
\end{equation}
Then
\begin{align}
	\mathrm{d}\Gamma_s &= \Gamma_s a_s \mathrm{d}s + \Gamma_s \langle c_s,\mathrm{d}W_s\rangle_{\Xi}\\
	\mathrm{d}\Gamma_s^{-1}& = \Gamma_s^{-1} (-a_s + \|c_s\|_{\Xi}^2)\mathrm{d}s - \Gamma_s^{-1} \langle c_s,\mathrm{d}W_s\rangle_{\Xi}.
\end{align}
By \cite[Proposition 6.18]{fabbri2017}, there exists an adapted process $V\in L^2([t,T]\times\Omega; \Xi)$ such that
\begin{equation}\label{representationtheorem}
	\Gamma_T \eta + \int_t^T \Gamma_s b_s \mathrm{d}s = \mathbb{E} \left [ \Gamma_T \eta + \int_t^T \Gamma_s b_s \mathrm{d}s \right ] + \int_t^T \langle V_s, \mathrm{d}W_s \rangle_{\Xi}.
\end{equation}
\begin{proposition}\label{linearequation}
	The solution of the linear BSDE \eqref{linearbsde} is given by
	\begin{equation}\label{formulaYZ}
	\begin{cases}
		Y_s = \Gamma_s^{-1} \mathbb{E} \left [ \Gamma_T \eta + \int_s^T \Gamma_r b_r \mathrm{d}r \middle | \mathcal{F}_s \right ]\\
		Z_s = \Gamma_s^{-1} V_s - c_s Y_s.
	\end{cases}
	\end{equation}
\end{proposition}

\begin{proof}
	Using \eqref{representationtheorem}, we derive from \eqref{formulaYZ} that
	\begin{align}
		Y_s &= \Gamma_s^{-1} \left ( \mathbb{E} \left [ \Gamma_T \eta + \int_t^T \Gamma_r b_r \mathrm{d}r \right ] - \int_t^s \Gamma_r b_r \mathrm{d}r + \int_t^s \langle V_r, \mathrm{d}W_r \rangle_{\Xi} \right )\\
		&= \Gamma_s^{-1} \left ( \mathbb{E} \left [ \Gamma_T \eta + \int_t^T \Gamma_r b_r \mathrm{d}r \right ] - \int_t^s \Gamma_r b_r \mathrm{d}r + \int_t^s \langle \Gamma_r Z_r + \Gamma_r c_r Y_r, \mathrm{d}W_r \rangle_{\Xi} \right ).
	\end{align}
	Applying It\^o's formula yields
	\begin{align}
		\mathrm{d}Y_s &= \left [ \Gamma_s^{-1} \left ( -a_s \mathrm{d}s + \|c_s\|_{\Xi}^2 \mathrm{d}s \right ) - \Gamma_s^{-1} \langle c_s,\mathrm{d}W_s \rangle_{\Xi} \right ] \Gamma_s Y_s\\
		&\qquad + \Gamma_s^{-1} [ \langle \Gamma_s Y_s c_s + \Gamma_s Z_s, \mathrm{d}W_s \rangle_{\Xi} - \Gamma_s b_s \mathrm{d}s ] - \Gamma_s^{-1} \langle c_s, c_s \Gamma_s Y_s +\Gamma_s Z_s \rangle_{\Xi} \mathrm{d}s\\
		&= -[a_s Y_s + b_s + \langle c_s,Z_s \rangle_{\Xi} ] \mathrm{d}s + \langle Z_s,\mathrm{d}W_s \rangle_{\Xi}.
	\end{align}
	Since $Y_T = \eta$, this proves that $(Y,Z)$ as given in \eqref{formulaYZ} solves the BSDE \eqref{linearbsde}.
\end{proof}

\subsection{Comparison for BSDEs}

In the remainder of this section, let $F:\Omega\times[t,T]\times \mathbb{R} \times \Xi \to \mathbb{R}$ be $\mathcal{P}\times \mathcal{B}(\mathbb{R}\times \Xi )/\mathcal{B}(\mathbb{R})$-measurable, where $\mathcal{P}$ denotes the progressive $\sigma$-algebra on $\Omega\times [t,T]$ and $\mathcal{B}(\Lambda)$ denotes the Borel $\sigma$-algebra of any topological space $\Lambda$. Furthermore, assume that there exists a constant $C>0$ such that
\begin{equation}
	| F(s,y,z) - F(s,y^{\prime},z^{\prime}) | \leq C(|y-y^{\prime}| + \|z-z^{\prime}\|_{\Xi} )
\end{equation}
$\mathbb{P}$-almost surely and for every $s\in [t,T]$, $y,y^{\prime}\in \mathbb{R}$ and $z,z^{\prime}\in \Xi$. Finally, assume that
\begin{equation}
	\mathbb{E} \left [ \int_t^T | F(s,0,0)|^2 \mathrm{d}s \right ] <\infty.
\end{equation}
We consider the BSDE
\begin{equation}\label{bsde1}
	\begin{cases}
		\mathrm{d}Y_s =F (s,Y_s,Z_s) \mathrm{d}s + \langle Z_s, \mathrm{d}W_s\rangle_{\Xi},\quad s\in [t,T]\\
		Y_T = \eta\in L^2(\Omega),
	\end{cases}
\end{equation}
where $\eta$ is $\mathcal{F}_T$-measurable.
\begin{definition}\label{definition_subsupersolution_1}
	A pair $(Y,Z)$ of adapted processes $Y\in L^2(\Omega;C([t,T]))$ and $Z\in L^2([t,T]\times\Omega; \Xi)$ is a supersolution of the BSDE \eqref{bsde1} if for every $t\leq s<r\leq T$ it holds
	\begin{equation}\label{definitionsupersolution1}
	\begin{cases}
		Y_r \leq Y_s +\int_s^r F({s^{\prime}},Y_{s^{\prime}},Z_{s^{\prime}}) \mathrm{d}{s^{\prime}} +\int_s^r \langle Z_{s^{\prime}},\mathrm{d}W_{s^{\prime}} \rangle_{\Xi}\\
		Y_T = \eta.
	\end{cases}
	\end{equation}
	A pair $(Y,Z)$ of adapted processes $Y\in L^2(\Omega;C([t,T]))$ and $Z\in L^2([t,T]\times\Omega; \Xi)$ is a subsolution of the BSDE \eqref{bsde1} if for every $t\leq s<r\leq T$ it holds
	\begin{equation}\label{definitionsubsolution1}
	\begin{cases}
		Y_r \geq Y_s +\int_s^r F({s^{\prime}},Y_{s^{\prime}},Z_{s^{\prime}}) \mathrm{d}{s^{\prime}} +\int_s^r \langle Z_{s^{\prime}},\mathrm{d}W_{s^{\prime}} \rangle_{\Xi}\\
		Y_T = \eta.
	\end{cases}
	\end{equation}
\end{definition}
Note that this definition is equivalent with the following definition.
\begin{definition}
	A triple $(Y,Z,I)$ of adapted processes $Y\in L^2(\Omega;C([t,T]))$, $Z\in L^2([t,T]\times\Omega; \Xi)$ and $I\in L^2(\Omega; C([t,T]))$ is a supersolution of the BSDE \eqref{bsde1} if $I(t)=0$ $\mathbb{P}$-almost surely, $I$ is increasing, and for every $s\in [t,T]$ it holds
	\begin{equation}\label{definitionsupersolution2}
		Y_s = \eta - \int_s^T F(r,Y_r,Z_r) \mathrm{d}r + I_T - I_s - \int_s^T \langle Z_r,\mathrm{d}W_r \rangle_{\Xi}.
	\end{equation}
	A triple $(Y,Z,D)$ of adapted processes $Y\in L^2(\Omega;C([t,T]))$, $Z\in L^2([t,T]\times\Omega; \Xi)$ and $D\in L^2(\Omega; C([t,T]))$ is a subsolution of the BSDE \eqref{bsde1} if $D(t)=0$ $\mathbb{P}$-almost surely, $D$ is decreasing, and for every $s\in [t,T]$ it holds
	\begin{equation}\label{definitionsubsolution2}
		Y_s = \eta - \int_s^T F(r,Y_r,Z_r) \mathrm{d}r + D_T - D_s - \int_s^T \langle Z_r,\mathrm{d}W_r \rangle_{\Xi}.
	\end{equation}
\end{definition}

\begin{remark}
	The fact that \eqref{definitionsupersolution1} and \eqref{definitionsubsolution1} follow from \eqref{definitionsupersolution2} and \eqref{definitionsubsolution2}, respectively, is immediate. For the opposite direction, consider the process
	\begin{equation}
		s \mapsto Y_t + \int_t^s F(r,Y_r,Z_r) \mathrm{d}r + \int_t^s \langle Z_r, \mathrm{d}W_r \rangle_{\Xi} - Y_s.
	\end{equation}
	Note that this process is increasing if $(Y,Z)$ is a supersolution in the sense of Definition \ref{definition_subsupersolution_1} and decreasing if $(Y,Z)$ is a subsolution in the sense of Definition \ref{definition_subsupersolution_1}. Moreover, using this process as $I$ and $D$, respectively, equations \eqref{definitionsupersolution2} and \eqref{definitionsubsolution2} are satisfied.
\end{remark}
In the case of linear BSDEs, we have the following comparison result.
\begin{proposition}\label{comparisonlinear}
	Let $(Y,Z)$ be a supersolution of the linear BSDE \eqref{linearbsde}. Then it holds
	\begin{equation}\label{domination}
		Y_s \geq \Gamma_s^{-1} \mathbb{E} \left [ \Gamma_T \eta + \int_s^T \Gamma_r b_r \mathrm{d}r \middle | \mathcal{F}_s \right ],
	\end{equation}
	i.e., a supersolution to the linear BSDE \eqref{linearbsde} dominates the solution to the linear BSDE.
\end{proposition}

\begin{proof}
	Let $(Y,Z,I)$ be a supersolution of the linear BSDE \eqref{linearbsde}. Then an application of It\^o's formula yields that
	\begin{equation}
		\Gamma_s Y_s + \int_t^s \Gamma_r b_r \mathrm{d}r
	\end{equation}
	is a local supermartingale. From this fact, one easily deduces \eqref{domination}.
\end{proof}
Next, let us prove a comparison result for sub- and supersolutions.
\begin{proposition}\label{comparison}
	Let $(F^i,\eta^i)$, $i=1,2$, satisfy the same assumptions as $(F,\eta)$ above, and let $(Y^1,Z^1)$ and $(Y^2,Z^2)$ be a sub- and supersolution of the BSDE \eqref{bsde1} associated with $(F^1,\eta^1)$ and $(F^2,\eta^2)$, respectively. Assume
	\begin{enumerate}[label=(\roman*)]
		\item $\eta^2 \geq \eta^1$
		\item $F^2(s,Y^2_s, Z^2_s) \leq F^1(s,Y^2_s, Z^2_s)$ $\mathrm{d}s\otimes\mathbb{P}$-almost surely.
	\end{enumerate}
	Then for every $s\in [t,T]$, it holds $Y^2_s \geq Y^1_s$ $\mathbb{P}$-almost surely. Moreover, if there exists an $s\in [t,T]$ such that
	\begin{equation}\label{strict}
		\eta^2 - \eta^1 + \int_{s}^T F^{1}(r, Y^2_r, Z^2_r) - F^{2}(r, Y^2_r, Z^2_r) \mathrm{d}r > 0
	\end{equation}
	$\mathbb{P}$-almost surely, then $Y^2_{s} > Y^1_{s}$ $\mathbb{P}$-almost surely.
\end{proposition}

\begin{proof}
	Note that
	\begin{equation}
		Y^2_r - Y^1_r \leq Y^2_s - Y^1_s - \int_s^r a_{s^{\prime}}(Y^2_{s^{\prime}} - Y^1_{s^{\prime}}) + b_{s^{\prime}} + \langle c_{s^{\prime}}, Z^2_{s^{\prime}}- Z^1_{s^{\prime}} \rangle_{\Xi} \mathrm{d}{s^{\prime}} + \int_s^r \langle Z^2_{s^{\prime}}-Z^1_{s^{\prime}}, \mathrm{d}W_{s^{\prime}}\rangle_{\Xi},
	\end{equation}
	where
	\begin{equation}
		a_{s^{\prime}} = \frac{1}{Y^2_{s^{\prime}} - Y^1_{s^{\prime}}}(F^1({s^{\prime}},Y^1_{s^{\prime}},Z^1_{s^{\prime}})-F^1({s^{\prime}},Y^2_{s^{\prime}},Z^1_{s^{\prime}})) \mathbf{1}_{\{Y^2_{s^{\prime}} \neq Y^1_{s^{\prime}}\}}
	\end{equation}
	and
	\begin{equation}
		b_{s^{\prime}} =  F^1({s^{\prime}},Y^2_{s^{\prime}},Z^2_{s^{\prime}})-F^2({s^{\prime}},Y^2_{s^{\prime}},Z^2_{s^{\prime}}) 
	\end{equation}
	and
	\begin{equation}
		c_{s^{\prime}} = \frac{Z^2_{s^{\prime}}-Z^1_{s^{\prime}}}{\|Z^2_{s^{\prime}} - Z^1_{s^{\prime}}\|_{\Xi}^2} (F^1({s^{\prime}},Y^2_{s^{\prime}},Z^1_{s^{\prime}})-F^1({s^{\prime}},Y^2_{s^{\prime}},Z^2_{s^{\prime}}) ) \mathbf{1}_{\{Z^2_{s^{\prime}}\neq Z^1_{s^{\prime}}\}}.
	\end{equation}
	Therefore, we obtain from Proposition \ref{comparisonlinear} and assumptions $(i)$ and $(ii)$
	\begin{equation}
		Y^2_s-Y^1_s \geq \Gamma_s^{-1} \mathbb{E} \left [ \Gamma_T (\eta^2- \eta^1) + \int_s^T \Gamma_r b_r \mathrm{d}r \middle | \mathcal{F}_s \right ]\geq 0,
	\end{equation}
	where $\Gamma$ is given by equation \eqref{gamma}. The strict inequality follows from \eqref{strict}.
\end{proof}

\subsection{A Priori Estimates for BSDEs}

\begin{lemma}\label{BSDEaprioriestimate}
	Let $(Y,Z)$ be the solution of the BSDE \eqref{bsde1}. Then it holds
	\begin{equation}
		\mathbb{E} \left [ \sup_{s\in [t,T]} |Y_s|^2 + \int_t^T \| Z_s \|_{\Xi}^2 \mathrm{d}s \right ] \leq C \mathbb{E} \left [ |\eta|^2 + \left ( \int_t^T |F(s,0,0)| \mathrm{d}s \right )^2 \right ].
	\end{equation}
\end{lemma}

The proof of this result can be found in \cite[Proposition 4.3]{fuhrman2002}.

\begin{lemma}\label{BSDEdependence}
	Let $(F^i,\eta^i)$, $i=1,2$, satisfy the same assumptions as $(F,\eta)$ above and let $(Y^1,Z^1)$ and $(Y^2,Z^2)$ be the solution of the BSDE \eqref{bsde1} associated with $(F^1,\eta^1)$ and $(F^2,\eta^2)$, respectively. Then it holds
	\begin{equation}
	\begin{split}
		&\mathbb{E} \left [ \sup_{s\in [t,T]} \left |Y^1_s - Y^2_s \right |^2 + \int_t^T \left  \| Z^1_s - Z^2_s \right \|_{\Xi}^2 \mathrm{d}s \right ]\\
		&\leq C\mathbb{E} \left [\left |\eta^1 - \eta^2\right |^2 + \left ( \int_t^T |F^1(r,Y^1_r,Z^1_r) - F^2(r,Y^1_r,Z^1_r)| \mathrm{d}r \right )^2 \right ].
	\end{split}
	\end{equation}
\end{lemma}

\begin{proof}
	Note that for all $s\in [t,T]$ it holds
	\begin{equation}
		\begin{split}
			Y^1_s - Y^2_s
			&= \eta^1 - \eta^2 - \int_s^T F^1(r, Y^1_r,Z^1_r ) - F^2(r, Y^2_r, Z^2_r ) \mathrm{d}r - \int_s^T \langle Z^1_r -  Z^2_r, \mathrm{d}W_r \rangle_{\Xi}\\
			&= \eta^1 - \eta^2 - \int_s^T F^1(r, Y^1_r,Z^1_r ) - F^2(r, Y^1_r,Z^1_r ) \mathrm{d}r\\
			&\qquad - \int_s^T a_r (Y^1_r - Y^2_r) + b_r (Z^1_r - Z^2_r) \mathrm{d}r - \int_s^T \langle Z^1_r - Z^2_r, \mathrm{d}W_r \rangle_{\Xi},
		\end{split}
	\end{equation}
	where
	\begin{equation}
		a_r := \frac{ F^2(r, Y^1_r,Z^1_r ) - F^2(r,  Y^2_r,Z^1_r )}{Y^1_r - Y^2_r} \mathbf{1}_{\{Y^1_r \neq Y^2_r\}}
	\end{equation}
	and
	\begin{equation}
		b_r := \frac{F^2(r, Y^2_r,Z^1_r ) - F^2(r,  Y^2_r, Z^2_r )}{Z^1_r - Z^2_r} \mathbf{1}_{\{Z^1_r \neq Z^2_r\}}.
	\end{equation}
	Now, Lemma \ref{BSDEaprioriestimate} yields the result.
\end{proof}

\section{Feynman--Kac Formulae for $B$-Continuous Viscosity Solutions}\label{mainresult}

In this section, we prove that the stochastic representation formula \eqref{stochasticrepresentation} yields a $B$-continuous viscosity solution of the PDE \eqref{PDE}. Under additional assumptions, we establish uniqueness.

\subsection{$B$-Continuous Viscosity Solutions}

First, we recall the notion of $B$-continuous viscosity solutions. Let us begin by introducing the notion of $B$-continuity and the class of test functions we are using, see \cite[Definitions 3.3, 3.4, and 3.32]{fabbri2017}.
\begin{definition}
	Let $B\in L(H)$ be a strictly positive, self-adjoint operator on $H$. A function $u:(0,T) \times H\to \mathbb{R} \cup \{\pm \infty \}$ is $B$-upper semicontinuous (respectively, $B$-lower semicontinuous) if, for any sequences $(t_n)_{n\in \mathbb{N}}$ in $(0,T)$ and $(x_n)_{n\in\mathbb{N}}$ in $H$ such that $t_n\to t\in (0,T)$, $x_n\rightharpoonup x$ and $B x_n \to Bx$ as $n\to \infty$, we have
	\begin{equation}
		\limsup_{n\to\infty} u(t_n,x_n) \leq u(t,x) \qquad (\text{respectively,} \, \liminf_{n\to\infty} u(t_n,x_n) \geq u(t,x) ).
	\end{equation}
	A function $u:(0,T)\times H \to \mathbb{R}$ is $B$-continuous if it is both $B$-upper semicontinuous and $B$-lower semicontinuous.
\end{definition}

\begin{definition}\label{definitiontestfunction}
	A function $\psi$ is a test function if $\psi = \varphi +h(t,\|x\|_H)$, where:
	\begin{enumerate}[label=(\roman*)]
		\item $\varphi \in C^{1,2}((0,T)\times H)$ is locally bounded, and is such that $\varphi$ is $B$-lower semicontinuous, and $\varphi_t$, $A^{\ast} D\varphi$, $D\varphi$, $D^2\varphi$ are uniformly continuous on $(0,T)\times H$.
		\item $h\in C^{1,2}((0,T)\times \mathbb{R})$ and is such that for every $t\in (0,T)$, $h(t,\cdot)$ is even and $h(t,\cdot)$ is non-decreasing on $[0,+\infty)$.
	\end{enumerate}
\end{definition}
Now, we define the notion of $B$-continuous viscosity solution, see \cite[Definition 3.35]{fabbri2017}.
\begin{definition}\label{definitionviscositysolution}
	A locally bounded and upper semicontinuous function $u$ on $(0,T]\times H$ which is $B$-upper semicontinuous on $(0,T)\times H$ is a viscosity subsolution of \eqref{PDE} if $u(T,x^{\prime})\leq g(x^{\prime})$ for $x^{\prime}\in H$ and the following holds: whenever $u-\psi$ has a local maximum at a point $(t,x)\in (0,T)\times H$ for a test function $\psi$ as in Definition \ref{definitiontestfunction} then
	\begin{multline}
		\psi_t(t,x) +\langle x,A^{\ast} D\varphi(t,x)\rangle_H + \langle b(t,x), D\psi(t,x) \rangle_H\\
		+ \frac12 \text{tr} (\sigma^{\ast}(t,x) D^2\psi(t,x) \sigma(t,x)) - f(t,x,u(t,x),\sigma^{\ast}(t,x)D\psi(t,x)) \geq 0.
	\end{multline}
	A locally bounded and lower semicontinuous function $u$ on $(0,T]\times H$ which is $B$-lower semicontinuous on $(0,T)\times H$ is a viscosity supersolution of \eqref{PDE} if $u(T,x^{\prime})\geq g(x^{\prime})$ for $x^{\prime}\in H$ and the following holds: whenever $u+\psi$ has a local minimum at a point $(t,x)\in (0,T)\times H$ for a test function $\psi$ as in Definition \ref{definitiontestfunction} then
	\begin{multline}
		- \psi_t(t,x) - \langle x,A^{\ast} D\varphi(t,x)\rangle_H - \langle b(t,x), D\psi(t,x) \rangle_H\\
		- \frac12 \text{tr} (\sigma^{\ast}(t,x) D^2\psi(t,x) \sigma(t,x)) - f(t,x,u(t,x),-\sigma^{\ast}(t,x)D\psi(t,x)) \leq 0.
	\end{multline}
	A viscosity solution of \eqref{PDE} is a function which is both a viscosity subsolution and a viscosity supersolution of \eqref{PDE}.
\end{definition}

\subsection{Existence and Stochastic Representation}

First, we recall the following a priori estimates for the forward SPDE \eqref{forwardequation}, see \cite[Lemma 3.23]{fabbri2017} and \cite[Proposition 3.3]{fuhrman2002}, respectively.

\begin{lemma}\label{apriori1}
	Let Assumptions \ref{assumptionA} and \ref{assumptionB} be satisfied. Then, there exists a constant $C>0$ independent of $t\in [0,T]$ such that
	\begin{equation}
		\mathbb{E} \left [ \|X^{t,x}_T - X^{t,x^{\prime}}_T \|_{H_{-1}}^2 + \int_t^T \| X^{t,x}_s - X^{t,x^{\prime}}_s \|_H^2 \mathrm{d}s \right ] \leq C \| x - x^{\prime} \|_{H_{-1}}^2
	\end{equation}
	for all $x,x^{\prime}\in H$. Furthermore, there exists a constant $C>0$ such that
	\begin{equation}
		\mathbb{E} \left [ \| X^{t,x}_T - X^{t,x^{\prime}}_T \|_{H}^2 \right ] \leq \frac{C}{T-t} \| x - x^{\prime} \|_{H_{-1}}^2
	\end{equation}
	for every $t\in[0,T)$ and for all $x,x^{\prime}\in H$.
\end{lemma}

For the proof of our main result, we need to extend the processes $(X^{t,x}_s)_{s\in [t,T]}$ and $(Y^{t,x}_s,Z^{t,x}_s)_{s\in [t,T]}$ to the entire interval $[0,T]$. For $s\in [0,t)$, we set $X^{t,x}_s := x$. Now, that $X^{t,x}$ is defined on the entire interval $[0,T]$, we can solve the BSDE \eqref{bsde} on the entire interval.

\begin{lemma}\label{apriori2}
	Let Assumptions \ref{assumptionA} and \ref{assumptionB} be satisfied. Then, for every $x\in H$, it holds
	\begin{equation}
		\mathbb{E} \left [ \sup_{s\in [0,T]} \| X^{t^{\prime},x}_s - X^{t,x}_s \|_{H}^2 \right ] \to 0
	\end{equation}
	as $t^{\prime} \to t$.
\end{lemma}
\begin{remark}
	In \cite[Proposition 3.3]{fuhrman2002}, the authors actually prove a stronger result than this. However, we only need the stated result.
\end{remark}
Now, let us state and prove our main result.

\begin{theorem}\label{viscosityproperty}
	Let Assumption \ref{assumptionA}, \ref{assumptionB} and \ref{assumptionC} be satisfied. Let $(Y,Z)$ be the solution of the BSDE \eqref{bsde}, and define $u(t,x) := Y^{t,x}_t$. Then $u$ is a $B$-continuous viscosity solution of the PDE \eqref{PDE} in the sense of Definition \ref{definitionviscositysolution}.
\end{theorem}

\begin{proof}
	Note that the terminal condition is satisfied by definition. In the first step, we are going to prove the $B$-continuity of $u$, and in the second step, we are going to prove the viscosity property.

	\emph{Step 1}: Let $(t,x) \in (0,T) \times H$. Let $t_n \to t$, $x_n\rightharpoonup x$ and $Bx_n\to Bx$ and fix $N\in\mathbb{N}$ such that $|t_n-t| < (T-t)/2$ for all $n\geq N$. We have
	\begin{equation}\label{triangleinequality}
		|u(t_n,x_n) -u(t,x)| = \left |Y^{t_n,x_n}_{t_n} - Y^{t,x}_t \right | \leq \left |Y^{t_n,x_n}_{t_n} - Y^{t_n,x}_{t_n}\right | + \left |Y^{t_n,x}_{t_n} - Y^{t,x}_t \right |.
	\end{equation}
	For the first term, choosing
	\begin{equation}
	\begin{cases}
		F^1(s,y,z) := f(s,X^{t_n,x_n}_s,y,z)\\
		\eta^1 := g(X^{t_n,x_n}_T)
	\end{cases}
	\end{equation}
	and
	\begin{equation}
	\begin{cases}
		F^2(s,y,z) := f(s,X^{t_n,x}_s,y,z)\\
		\eta^2 := g(X^{t_n,x}_T)
	\end{cases}
	\end{equation}
	and applying Lemma \ref{BSDEdependence} yields
	\begin{equation}
	\begin{split}
		&\left |Y^{t_n,x_n}_{t_n} - Y^{t_n,x}_{t_n} \right |^2\\
		 &\leq C \mathbb{E} \Bigg [ \left |g(X^{t_n,x_n}_T) - g(X^{t_n,x}_T)\right |^2\\
		 &\qquad\qquad + \left ( \int_{t_n}^T |f(s,X^{t_n,x_n}_s,Y^{t_n,x_n}_s,Z^{t_n,x_n}_s) - f(s,X^{t_n,x}_s,Y^{t_n,x_n}_s,Z^{t_n,x_n}_s)| \mathrm{d}s \right )^2 \Bigg ].
	\end{split}
	\end{equation}
	Using Assumption \ref{assumptionC}, we obtain
	\begin{equation}
		\left |Y^{t_n,x_n}_{t_n} - Y^{t_n,x}_{t_n} \right |^2 \leq C\mathbb{E} \left [\|X^{t_n,x_n}_T - X^{t_n,x}_T\|_{H}^2 + \int_{t_n}^T \|X^{t_n,x_n}_s - X^{t_n,x}_s\|_H^2 \mathrm{d}s \right ].
	\end{equation}
	Applying Lemma \ref{apriori1} yields
	\begin{equation}\label{step1bcontinuity}
		\left |Y^{t_n,x_n}_{t_n} - Y^{t_n,x}_{t_n} \right |^2 \leq C \| x_n - x \|_{H_{-1}}^2
	\end{equation}
	for some constant that is uniform in $n\geq N$.
	
	Now let us turn to the second term in \eqref{triangleinequality}. Let us first consider the case that $t_n \downarrow t$. Then we have
	\begin{equation}\label{secondterm}
		\begin{split}
		&\left |Y^{t_n,x}_{t_n} - Y^{t,x}_t \right |^2\\
		&\leq C \mathbb{E} \Bigg [ \left | g(X^{t_n,x}_T) - g(X^{t,x}_T) \right |^2 + \int_{t_n}^T |f(s,X^{t_n,x}_s,Y^{t_n,x}_s,Z^{t_n,x}_s) - f(s,X^{t,x}_s,Y^{t,x}_s,Z^{t,x}_s)|^2 \mathrm{d}s\\
		&\quad + \left | \int_{t_n}^T \langle Z^{t_n,x}_s - Z^{t,x}_s, \mathrm{d}W_s \rangle_{\Xi} \right |^2 + \int_{t}^{t_n} | f(s,X^{t,x}_s, Y^{t,x}_s,Z^{t,x}_s) |^2 \mathrm{d}s + \left | \int_{t}^{t_n} \langle Z^{t,x}_s ,\mathrm{d}W_s \rangle_{\Xi} \right |^2 \Bigg ].
		\end{split}
	\end{equation}
	The fourth and the fifth term tend to zero as $t_n \downarrow t$ due to the continuity of the integral. For the first term, we have
	\begin{equation}\label{lipschitzg}
		\left | g(X^{t_n,x}_T) - g(X^{t,x}_T) \right |^2 \leq C \| X^{t_n,x}_T - X^{t,x}_T \|^2_{H}.
	\end{equation}
	For the second term we have by Assumption \ref{assumptionf}
	\begin{equation}
	\begin{split}
		&\int_{t_n}^T |f(s,X^{t_n,x}_s,Y^{t_n,x}_s,Z^{t_n,x}_s) - f(s,X^{t,x}_s,Y^{t,x}_s,Z^{t,x}_s)|^2 \mathrm{d}s \\
		&\leq C \int_{t_n}^T \| X^{t_n,x}_s -  X^{t,x}_s \|^2_H + | Y^{t_n,x}_s-Y^{t,x}_s |^2 + \| Z^{t_n,x}_s - Z^{t,x}_s \|^2_{\Xi} \mathrm{d}s.
	\end{split}
	\end{equation}
	Therefore, applying Lemma \ref{BSDEdependence} and using \eqref{lipschitzg}, we obtain
	\begin{equation}
	\begin{split}
		&\mathbb{E} \left [ \int_{t_n}^T |f(s,X^{t_n,x}_s,Y^{t_n,x}_s,Z^{t_n,x}_s) - f(s,X^{t,x}_s,Y^{t,x}_s,Z^{t,x}_s)|^2 \mathrm{d}s \right ]\\ &\leq C \mathbb{E} \left [ \sup_{s\in [0,T]} \|X^{t_n,x}_s - X^{t,x}_s \|_{H}^2 \right ].
	\end{split}
	\end{equation}
	For the third term, we have by Lemma \ref{BSDEdependence}
	\begin{equation}
	\begin{split}
		\mathbb{E} \left [ \left | \int_{t_n}^T \langle Z^{t_n,x}_s - Z^{t,x}_s, \mathrm{d}W_s \rangle_{\Xi} \right |^2 \right ] &= \mathbb{E} \left [ \int_{t_n}^T \| Z^{t_n,x}_s - Z^{t,x}_s \|_{\Xi}^2 \mathrm{d}s \right ]\\
		&\leq C \mathbb{E} \left [ \sup_{s\in [0,T]} \|X^{t_n,x}_s - X^{t,x}_s \|_{H}^2 \right ].
	\end{split}
	\end{equation}
	Applying Lemma \ref{apriori2} yields that the right-hand side of \eqref{secondterm} tends to zero as $t_n\downarrow t$. Now, let $t_n\uparrow t$. In this case, \eqref{secondterm} takes the form
	\begin{equation}\label{secondterm_2}
		\begin{split}
			&\left |Y^{t_n,x}_{t_n} - Y^{t,x}_t \right |^2\\
			&\leq C \mathbb{E} \Bigg [ \left | g(X^{t_n,x}_T) - g(X^{t,x}_T) \right |^2 + \int_{t}^T |f(s,X^{t_n,x}_s,Y^{t_n,x}_s,Z^{t_n,x}_s) - f(s,X^{t,x}_s,Y^{t,x}_s,Z^{t,x}_s)|^2 \mathrm{d}s\\
			&\quad + \left | \int_{t}^T \langle Z^{t_n,x}_s - Z^{t,x}_s, \mathrm{d}W_s \rangle_{\Xi} \right |^2 + \int_{t_n}^{t} | f(s,X^{t_n,x}_s, Y^{t_n,x}_s,Z^{t_n,x}_s) |^2 \mathrm{d}s + \left | \int_{t_n}^{t} \langle Z^{t_n,x}_s ,\mathrm{d}W_s \rangle_{\Xi} \right |^2 \Bigg ].
		\end{split}
	\end{equation}
	The first three terms can be treated with similar arguments as before. For the fourth term, we have
	\begin{equation}\label{newestimate}
	\begin{split}
		&\mathbb{E} \left [ \int_{t_n}^{t} | f(s,X^{t_n,x}_s, Y^{t_n,x}_s,Z^{t_n,x}_s) |^2 \mathrm{d}s \right ]\\
		&\leq C \int_{t_n}^t | f(s,0,0,0) |^2 \mathrm{d}s + C \mathbb{E} \left [ \int_{t_n}^t \| X_s^{t_n,x} \|_H^2 + |Y^{t_n,x}_s|^2 + \| Z^{t_n,x}_s \|_{\Xi}^2 \mathrm{d}s \right ].
	\end{split}
	\end{equation}
	Due to Assumption \ref{assumptionf}, the first term converges to zero as $t_n\uparrow t$. Considering the term involving $X^{t_n,x}$ in the second integral, we have
	\begin{equation}
	\begin{split}
		\mathbb{E} \left [ \int_{t_n}^t \| X^{t_n,x}_s \|_H^2 \mathrm{d}s \right ] &\leq C \mathbb{E} \left [ \int_{t_n}^t \| X^{t_n,x}_s - X^{t,x}_s \|_H^2 + \| X^{t,x}_s \|_H^2 \mathrm{d}s \right ]\\
		&\leq C \mathbb{E} \left [ \sup_{s\in [0,T]} \| X^{t_n,x}_s - X^{t,x}_s \|_H^2 \right ] + C \mathbb{E} \left [ \int_{t_n}^t \| X^{t,x}_s \|_H^2 \mathrm{d}s \right ]
	\end{split}
	\end{equation}
	where the first term on the right-hand side tends to zero as $t_n \uparrow t$ due to Lemma \ref{apriori2} and the second term tends to zero since the integrand is integrable over $[0,T]$. For the terms involving $Y^{t_n,x}$ and $Z^{t_n,x}$ in the second integral in \eqref{newestimate}, we have
	\begin{equation}\label{newestimate2}
	\begin{split}
		&\mathbb{E} \left [ \int_{t_n}^t |Y^{t_n,x}_s|^2 + \| Z^{t_n,x}_s \|_{\Xi}^2 \mathrm{d}s \right ] \\
		%&\leq C \mathbb{E} \left [ \int_{t_n}^t |Y^{t_n,x}_s - Y^{t,x}_s |^2 + |Y^{t,x}_s|^2+ \| Z^{t_n,x}_s - Z^{t,x}_s \|_{\Xi}^2 + \| Z^{t,x}_s \|_{\Xi}^2 \mathrm{d}s \right ]\\
		&\leq C\mathbb{E} \left [ \int_{0}^T |Y^{t_n,x}_s - Y^{t,x}_s |^2 + \| Z^{t_n,x}_s - Z^{t,x}_s \|_{\Xi}^2 \mathrm{d}s \right ] + \mathbb{E} \left [ \int_{t_n}^t |Y^{t,x}_s|^2 + \| Z^{t,x}_s \|_{\Xi}^2 \mathrm{d}s \right ].
	\end{split}
	\end{equation}
	Applying Lemma \ref{BSDEdependence} and Assumption \ref{assumptionC}, we obtain for the first term
	\begin{equation}
		\mathbb{E} \left [ \int_{0}^T |Y^{t_n,x}_s - Y^{t,x}_s |^2 + \| Z^{t_n,x}_s - Z^{t,x}_s \|_{\Xi}^2 \mathrm{d}s \right ]
		%&\leq C \mathbb{E} \left [ \left |g(X^{t_n,x}_T) - g(X^{t,x}_T)\right |^2 + \int_0^T \left | f(s,X^{t_n,x}_s,Y^{t,x}_s,Z^{t,x}_s) - f(s,X^{t,x}_s,Y^{t,x}_s,Z^{t,x}_s) \right |^2 \mathrm{d}s \right ]\\
		\leq C \mathbb{E} \left [ \sup_{s\in [0,T]} \| X^{t_n,x}_s - X^{t,x}_s \|_H^2 \right ],
	\end{equation}
	which tends to zero as $t_n \uparrow t$ due to Lemma \ref{apriori2}. The second term on the right-hand side of \eqref{newestimate2} tends to zero as $t_n\uparrow t$ since the integrand is integrable over $[0,T]$. Together with \eqref{step1bcontinuity}, this concludes the proof of the $B$-continuity.
	
	\emph{Step 2}: First, we show that $u$ is a viscosity subsolution. Let $\psi= \varphi+h$ be a test function as in Definition \ref{definitiontestfunction} such that
		\begin{equation}\label{maximum}
			0=(u-\psi)(t,x) = \max_{\substack{s\in [0,T]\\x^{\prime}\in H}} (u-\psi)(s,x^{\prime})
		\end{equation}
		We can assume without loss of generality that the maximum is strict, see \cite[Lemma 3.37]{fabbri2017}. Assume, for sake of contradiction, that
		\begin{multline}
			\psi_t(t,x) + \langle x, A^{\ast} D\varphi(t,x)\rangle_H + \langle b(t,x),D\psi(t,x) \rangle_H\\
			 + \frac12 \text{tr}(\sigma^{\ast}(t,x) D^2\psi(t,x) \sigma(t,x)) - f(t,x,u(t,x), \sigma^{\ast} D\psi(t,x)) <0.
		\end{multline}
		Fix $\varepsilon>0$ such that for all $t\leq s\leq t+\varepsilon$ and $\|x^{\prime}-x\|_H\leq \varepsilon$, it holds
		\begin{multline}\label{inequality}
			\psi_t(s,x^{\prime}) + \langle x^{\prime}, A^{\ast} D\varphi(s,x^{\prime})\rangle_H + \langle b(s,x^{\prime}),D\psi(s,x^{\prime}) \rangle_H\\
			+ \frac12 \text{tr}(\sigma^{\ast}(s,x^{\prime}) D^2\psi(s,x^{\prime}) \sigma(s,x^{\prime})) - f(s,x^{\prime},u(s,x^{\prime}),\sigma^{\ast} D\psi (s,x^{\prime})) < 0.
		\end{multline}
		Now, we define the stopping time
		\begin{equation}
			\tau := \inf \{ s\geq t | \|X^{t,x}_s -x\|_H \geq \varepsilon \} \wedge (t+\varepsilon).
		\end{equation}
		Let $(Y^{t,x},Z^{t,x})$ be the solution of the BSDE \eqref{bsde}, and define for $t\leq s \leq t+\varepsilon$
		\begin{equation}
			(Y^1_s, Z^1_s) := (Y^{t,x}_{s\wedge \tau} , \mathbf{1}_{[0,\tau]} (s) Z^{t,x}_{s}).
		\end{equation}
		Then we have
		\begin{equation}
			u(t+h,X^{t,x}_{t+h}) = Y^{t+h,X^{t,x}_{t+h}}_{t+h} = Y^{t,x}_{t+h}
		\end{equation}
		due to the uniqueness of the solution of the SPDE \eqref{forwardequation} and the uniqueness of the solution of the BSDE \eqref{bsde}. Therefore, for every $r\in [t,t+\varepsilon]$
		\begin{equation}
			Y^{t,x}_{r\wedge\tau} = u(r\wedge\tau,X^{t,x}_{r\wedge\tau}),
		\end{equation}
		which shows that $(Y^1,Z^1)$ solves the BSDE
		\begin{equation}\label{bsdeybar}
			\begin{cases}
				\mathrm{d}Y^1_s = \mathbf{1}_{[0,\tau]}(s) f(s,X^{t,x}_s, u(s,X^{t,x}_s) , Z^1_s) \mathrm{d}s + \langle Z^1_s, \mathrm{d}W_s\rangle_{\Xi} \\
				Y^1_{\tau} = u(\tau,X^{t,x}_{\tau}).
			\end{cases}
		\end{equation}
		On the other hand, we obtain from It\^o's formula for test functions (see \cite[Proposition 1.166]{fabbri2017}) that for every $r,s\in [t,t+\varepsilon]$, $r\geq s$,
		\begin{equation}
			\begin{split}
				&\psi(r\wedge \tau,X^{t,x}_{r\wedge\tau})\\
				&\leq \psi(s\wedge \tau,X^{t,x}_{s\wedge \tau}) + \int_s^{r} \mathbf{1}_{[0,\tau]}(s^{\prime}) \left ( \psi_t(s^{\prime},X^{t,x}_{s^{\prime}}) + \langle X^{t,x}_{s^{\prime}} , A^{\ast} D\varphi(s^{\prime},X^{t,x}_{s^{\prime}}) \rangle_H \right ) \mathrm{d}s^{\prime}\\
				&\quad + \int_s^{r} \mathbf{1}_{[0,\tau]}(s^{\prime}) \langle b(s^{\prime},X^{t,x}_{s^{\prime}}) , D\psi(s^{\prime},X^{t,x}_{s^{\prime}}) \rangle_H \mathrm{d}s^{\prime}\\
				&\quad + \frac12 \int_s^{r} \mathbf{1}_{[0,\tau]}(s^{\prime}) \text{tr}\left (\sigma^{\ast}(s^{\prime},X^{t,x}_{s^{\prime}}) D^2 \psi(s^{\prime},X^{t,x}_{s^{\prime}}) \sigma(s^{\prime},X^{t,x}_{s^{\prime}})\right ) \mathrm{d}s^{\prime}\\
				&\quad + \int_s^{r} \langle  \mathbf{1}_{[0,\tau]}(s^{\prime}) \sigma^{\ast}(s^{\prime},X^{t,x}_{s^{\prime}}) D\psi(s^{\prime},X^{t,x}_{s^{\prime}}), \mathrm{d}W_{s^{\prime}}\rangle_{\Xi} .
			\end{split}
		\end{equation}
		Therefore,
		\begin{equation}
			( Y^2_s, Z^2_s) := (\psi(s\wedge \tau,X^{t,x}_{s\wedge \tau}), \mathbf{1}_{[0,\tau]}(s) \sigma^{\ast}D\psi(s,X^{t,x}_s) )
		\end{equation}
		is a supersolution of the BSDE
		\begin{equation}\label{bsdeytilde}
			\begin{cases}
				\mathrm{d} Y^2_s = \mathbf{1}_{[0,\tau]}(s) \Bigl [ \psi_t(s,X^{t,x}_s) + \langle X^{t,x}_s , A^{\ast} D\varphi(s,X^{t,x}_s) \rangle_H + \langle b(s,X^{t,x}_s) , D\psi(s,X^{t,x}_s) \rangle_H\\
				\qquad\qquad\qquad\qquad\qquad + \frac12 \text{tr}(\sigma^{\ast}(s,X^{t,x}_s) D^2 \psi(s,X^{t,x}_s) \sigma(s,X^{t,x}_s)) \Bigr ] \mathrm{d}s + \langle Z^2_s, \mathrm{d}W_s\rangle_{\Xi} \\
				Y^2_{\tau} = \psi(\tau,X^{t,x}_{\tau}).
			\end{cases}
		\end{equation}
		Due to inequality \eqref{inequality} and the definition of $\tau$, we can apply \Cref{comparison} to obtain
		\begin{equation}
			\psi(t,x) = Y^2_t > Y^1_t = Y^{t,x}_t = u(t,x)
		\end{equation}
		which contradicts \eqref{maximum} and therefore concludes the proof that $u$ is a $B$-continuous viscosity subsolution. The proof that $u$ is a $B$-continuous viscosity supersolution is similar.
\end{proof}

\subsection{Uniqueness}

In order to prove uniqueness, we rely on a comparison theorem for the PDE \eqref{PDE} which is proved using analytic methods. For this result to hold, we have to impose additional assumptions on the coefficients of the PDE.

\begin{assumption}\label{assumptionD}
	\begin{enumerate}[label=(C\arabic*)]
		\item Let $\sigma$ be uniformly continuous on bounded subsets of $(0,T) \times H$.
		\item For $y^{\prime} \geq y$, let
		\begin{equation}
			f(s,x,y^{\prime},z) - f(s,x,y,z) \geq 0
		\end{equation}
		for all $(s,x,z)\in (0,T)\times H\times \Xi$.
	\end{enumerate}
\end{assumption}

\begin{theorem}
	Let Assumptions \ref{assumptionA}, \ref{assumptionB}, \ref{assumptionC} and \ref{assumptionD} be satisfied. Then $u(t,x) := Y^{t,x}_t$ is the unique $B$-continuous viscosity solutions satisfying
	\begin{equation}\label{terminalcondition}
		\lim_{t\to T} | u(t,x) - g(S(T-t) x) | = 0
	\end{equation}
	uniformly on bounded subsets of $H$ and
	\begin{equation}\label{ugrowth}
		|u(t,x)| \leq C_1 \exp \left ( C_2 \left ( \ln \left ( 1+\|x\|_H \right ) \right )^2 \right )
	\end{equation}
	for some constants $C_1,C_2\geq 0$ and all $(t,x)\in (0,T)\times H$.
\end{theorem}

\begin{proof}
	Once we know that $u$ satisfies \eqref{terminalcondition} and \eqref{ugrowth}, the result follows from \cite[Theorem 3.54]{fabbri2017}. Let us first prove \eqref{terminalcondition}. Note that
	\begin{equation}\label{triangleinequality2}
	\begin{split}
		&| u(t,x) - g(S(T-t)x) |^2\\
		& = \mathbb{E} \left [ | Y^{t,x}_t - g(S(T-t)x) |^2 \right ]\\
		&= \mathbb{E} \left [ \left | g(X^{t,x}_T) - \int_t^T f(s,X^{t,x}_s,Y^{t,x}_s,Z^{t,x}_s) \mathrm{d}s - \int_t^T \langle Z^{t,x}_s, \mathrm{d}W_s\rangle_{\Xi} - g(S(T-t)x) \right |^2 \right ]\\
		&\leq \mathbb{E} \left [ \int_t^T | f(s,X^{t,x}_s,Y^{t,x}_s,Z^{t,x}_s) |^2 \mathrm{d}s + \left | \int_t^T \langle Z^{t,x}_s, \mathrm{d}W_s \rangle_{\Xi} \right |^2 \right ]\\
		&\quad + \mathbb{E} \left [ | g(X^{t,x}_T) - g(S(T-t)x) |^2 \right ] .
	\end{split}
	\end{equation}
	For the first term, using It\^o's isometry and Assumption \ref{assumptionC}, we have
	\begin{equation}
	\begin{split}
		& \mathbb{E} \left [ \int_t^T | f(s,X^{t,x}_s,Y^{t,x}_s,Z^{t,x}_s) |^2 \mathrm{d}s + \left | \int_t^T \langle Z^{t,x}_s, \mathrm{d}W_s\rangle_{\Xi} \right |^2 \right ]\\
		&\leq C \int_t^T |f(s,0,0,0)|^2 + \mathbb{E} \left [ \|X^{t,x}_s\|_H^2 + |Y^{t,x}_s|^2 + \|Z^{t,x}_s \|_{\Xi}^2 \right ] \mathrm{d}s.
	\end{split}
	\end{equation}
	Applying Lemma \ref{BSDEaprioriestimate}, we obtain
	\begin{equation}\label{firststep}
	\begin{split}
		& \mathbb{E} \left [ \int_t^T | f(s,X^{t,x}_s,Y^{t,x}_s,Z^{t,x}_s) |^2 \mathrm{d}s + \left | \int_t^T \langle Z^{t,x}_s ,\mathrm{d}W_s \rangle_{\Xi} \right |^2 \right ]\\
		&\leq C \int_t^T |f(s,0,0,0)|^2 + 1 + \mathbb{E} \left [ \|X^{t,x}_s\|_H^2 \right ] \mathrm{d}s.
	\end{split}
	\end{equation}
	For the last term in \eqref{triangleinequality2}, using Assumption \ref{assumptiong}, we have
	\begin{equation}
	\begin{split}
		&\mathbb{E} \left [ | g(X^{t,x}_T) - g(S(T-t)x) |^2 \right ]\\
		&\leq C \mathbb{E} \left [ \| X^{t,x}_T - S(T-t)x \|_H^2 \right ]\\
		&= C \mathbb{E} \left [ \left \| \int_t^T S(T-s) b(s,X^{t,x}_s) \mathrm{d}s + \int_t^T S(T-s) \sigma(s,X^{t,x}_s) \mathrm{d}W_s \right \|_H^2 \right ]\\
		&\leq C \mathbb{E} \left [ \int_t^T \| S(T-s) b(s,X^{t,x}_s) \|_H^2 \mathrm{d}s + \left \| \int_t^T S(T-s) \sigma(s,X^{t,x}_s) \mathrm{d}W_s \right \|_H^2 \right ].
	\end{split}
	\end{equation}
	Using again It\^o's isometry and the linear growth assumption on $b$ and $\sigma$, we obtain
	\begin{equation}
		\mathbb{E} \left [ | g(X^{t,x}_T) - g(S(T-t)x) |^2 \right ] \leq C \int_t^T 1 + \mathbb{E} \left [ \|X^{t,x}_s \|_H^2 \right ] \mathrm{d}s.
	\end{equation}
	Therefore, together with \eqref{firststep}, we have
	\begin{equation}
		| u(t,x) - g(S(T-t)x) |^2 \leq C \int_t^T |f(s,0,0,0)|^2 + 1 + \mathbb{E} \left [ \|X^{t,x}_s\|_H^2 \right ] \mathrm{d}s.
	\end{equation}
	Due to Assumption \ref{assumptionf} and well-known a priori estimates for the solution of the SPDE \eqref{forwardequation} (see e.g. \cite[Theorem 7.2]{daprato2014}), the right-hand side converges to zero as $t\to T$ uniformly on bounded subsets of $H$. The growth condition \eqref{ugrowth} follows from Lemma \ref{BSDEaprioriestimate} and the same a priori estimates for the solution of the SPDE \eqref{forwardequation}.
\end{proof}
\iffalse
\section*{Declaration of Competing Interests}

The author declares that he has no known competing financial interests or personal relationships that could have appeared to influence the work reported in this paper.
\fi
\section*{Acknowledgments}
The author would like to thank Andrzej {\'{S}}wi{\k{e}}ch for his constructive criticism of the manuscript, as well as the anonymous referees for their feedback which helped to improve the manuscript. During the preparation of this work, the author was supported by a fellowship of the German Academic Exchange Service (DAAD).

\end{document}